\newif\ifTwoColumn
\newif\ifTechReport

\pdfoutput=1

\TechReporttrue

\ifTwoColumn
\documentclass[letterpaper, 10 pt, conference]{ieeeconf}

\else
\ifTechReport
\documentclass[review, 11 pt, a4paper, onehalfspacing, authoryear]{elsarticle}
\makeatletter
\def\ps@pprintTitle{%
 \let\@oddhead\@empty
 \let\@evenhead\@empty
 \def\@oddfoot{\centerline{\thepage}}%
 \let\@evenfoot\@oddfoot}
\makeatother

\else
\documentclass[12pt,draft,onecolumn,letterpaper]{ieeetrans}

\fi
\fi

\ifTwoColumn
\fi

\usepackage{amsmath,amsthm,amssymb}
\usepackage{bm}
\usepackage{graphicx}
\usepackage{float}
\usepackage{bbm}
\usepackage{mathtools}
\PassOptionsToPackage{hyphens}{url}\usepackage{url}
\usepackage{epstopdf}
\usepackage{xcolor}

\allowdisplaybreaks
\newtheorem{theorem}{Theorem}
\newtheorem{lemma}[theorem]{Lemma}
\newtheorem{assumption}{Assumption}
\newtheorem{proposition}[theorem]{Proposition}
\newtheorem{definition}{Definition}

\newcommand{\mbb}{\mathbb}
\renewcommand{\Re}{\mbb{R}}

\newcommand{\set}[2]{\left\{ #1\ \left| \ #2 \right. \right\}}

\usepackage[hidelinks]{hyperref}
\hypersetup{
colorlinks=false 
}
\usepackage{flushend}
\usepackage{enumerate}
\usepackage{setspace}

\usepackage{caption}
\usepackage{subcaption}
\biboptions{authoryear}

\begin{document}

\begin{frontmatter}

\title{Dynamic Programming for Optimal Delivery Time Slot Pricing\tnoteref{mytitlenote}}
\tnotetext[mytitlenote]{This work was supported by SIA Food Union Management.}
\author[addr]{Denis Lebedev\corref{corr}}
\cortext[corr]{Corresponding author}
\ead{denis.lebedev@eng.ox.ac.uk}
\author[addr]{Paul Goulart}
\ead{paul.goulart@eng.ox.ac.uk}
\author[addr]{Kostas Margellos}
\ead{kostas.margellos@eng.ox.ac.uk}
\address[addr]{Department of Engineering Science, University of Oxford, Oxford, OX1 3PJ, United Kingdom}
\begin{abstract}
	We study the dynamic programming approach to revenue management in the context of attended home delivery. We draw on results from dynamic programming theory for Markov decision problems to show that the underlying Bellman operator has a unique fixed point. We then provide a closed-form expression for the resulting fixed point and show that it admits a natural interpretation. Moreover, we also show that -- under certain technical assumptions -- the value function, which has a discrete domain and a continuous codomain, admits a continuous extension, which is a finite-valued, concave function of its state variables, at every time step. This result opens the road for achieving scalable implementations of the proposed formulation in future work, as it allows making informed choices of basis functions in an approximate dynamic programming context. We illustrate our findings on a simple numerical example and provide suggestions on how our results can be exploited to obtain closer approximations of the exact value function.
\end{abstract}

\begin{keyword}
Dynamic programming \sep Revenue management \sep Discrete convex analysis
\end{keyword}

\end{frontmatter}


\section{Introduction}\label{sec:intro}
The expenditure of US households on online grocery shopping could reach \$100 billion in 2022 according to the \citet{FMI2018}. Although growth forecasts vary and more conservative estimates lie, for example, at \$30 billion for the year 2021 \citep{PITCHBOOK2017}, the overall trend is clear: The online grocery sector is likely to grow if some of its main challenges can be overcome. 

One of these challenges is managing the logistics as one of the main cost-drivers. In particular, one can seek to exploit the flexibility of customers by offering delivery options at different prices to create delivery schedules that can be executed in a cost-efficient manner. To achieve this, recent proposals include giving customers the choice between narrow delivery time windows for high prices and \textit{vice versa} \citep{CAMPBELL2006} or charging customers different prices based on the area and their preferred delivery time \citep{ASDEMIR2009,YANG2016,YANG2017}. 

In this paper, we focus on the latter. We refer to the problem of finding the profit-maximising delivery slot prices as the \emph{revenue management problem in attended home delivery}, where ``attended'' refers to the requirement that customers need to be present upon delivery of the typically perishable goods, which is in contrast to, for example, standard mail delivery. Note that attended home delivery problems are more complex than standard delivery services, since goods need to be delivered in time windows that are pre-agreed with the customers. 

We adopt a dynamic programming (DP) model of an expected profit-to-go function, the value function of the DP, given the current state of orders and time left for customers to book a delivery slot. This DP was initially devised in the fashion industry \citep{GALLEGO1994}, but subsequently adopted and refined by the transportation sector and the attended home delivery industry \citep{YANG2016}. 

To find the optimal delivery slot prices, we need to compute the value function (at least approximately) for all states and times. The main challenge is that the state space of the DP grows exponentially with the set of delivery time slots, i.e.\ it suffers from the ``curse of dimensionality''. This means that for industry-sized problems, due to the prohibitively large number of states, the value function cannot be computed exactly, even off-line. Our ultimate objective is to compute improved value function approximations. Therefore, we study in this paper how the value function of the exact DP behaves mathematically in time and across state variables.

We show that the underlying DP operator has a unique fixed point. We then provide a closed-form expression of the resulting fixed point and derive a natural interpretation. Furthermore, we show that -- under certain technical assumptions -- for all time steps in the dynamic program, the value function admits a continuous extension, which is a finite-valued, concave function of its state variables. This result opens the road for achieving scalable implementations of the proposed formulation in future work, as it becomes possible to make informed choices of basis functions in an approximate dynamic programming context.

Improved value function approximations could finally be used for calculating optimal delivery slot prices. This has been shown by \citet{DONGETAL2009}, who prove that a unique set of optimal delivery slot prices exists, which can be found using Newton root search algorithms or using the Lambert-$W$ function as shown in \ref{ap:uncon} if estimates of the value function are known for all states and times.

Our paper is structured as follows: In the remainder of Section \ref{sec:intro}, we introduce some notation. In Section \ref{sec:rm_problem_form}, we define the revenue management problem in attended home delivery and its DP formulation. In Section \ref{sec:res}, we present our main results, Theorem \ref{th:fixedpoint}, which analytically characterises the fixed point of the DP, and Theorem \ref{th:v_conc}, which shows that there exists a continuous extension of the value function that is a finite-valued, concave function in its state variables at every time step. Section \ref{sec:proofs} contains reformulations of the DP into mathematically more convenient forms and develops supporting results leading to the proofs of the main results. Section \ref{sec:ex} presents a numerical illustration of the proposed scheme, while Section \ref{sec:conc} concludes the paper and suggests directions for future research. The Appendix contains the proofs of results not included in the main body of the paper.

\emph{Notation:}
Let	$\mathbf{1}$ denote a vector with all elements equal to $1$. Given some $s$, let $1_s$ be a vector of all zeros apart from the $s$-th entry, which equals 1. Furthermore, we define the convention that $1_0$ is a vector of zeros. Let $\Re_{+(+)}$ be the non-negative (positive) real numbers, let $\mathbb{Z}$ be the integers and let dim$(\cdot)$ denote the dimension of its argument. Let conv$(\cdot)$ denote the convex hull of its argument. We say that a function exhibits a monotonic behaviour if the monotonicity property holds element-wise, e.g.\ a function $f: \Re^N \mapsto \Re$ is monotonically increasing over its domain if $f(y) > f(x)$ for all $(x,y)$, such that at least one element of $y$ is greater than the corresponding element of $x$.
\section{Revenue Management Problem Formulation}\label{sec:rm_problem_form}
In this section, we derive a discrete-state formulation of the revenue management problem in attended home delivery. 
\subsection{Problem Statement}
We model an online business that delivers goods to locations of known customers. We consider a local approximation of the revenue management problem by dividing the service area geographically into a set of non-overlapping rectangular sub-areas, where the customers in each sub-area operate independently by being served by one delivery vehicle. This model resembles the setting in the work of \citet{YANG2017}. Due to this independence, we only consider a single sub-area, while our development directly extends to the case of multiple sub-areas. To cover all sub-areas in practice, it is possible to simply replicate our approach for every delivery sub-area, which would increase computational complexity linearly in the number of sub-areas, but which is easily parallelised.

We consider a finite booking horizon with possibly unequally-spaced time steps indexed by $t \in T:=\{1,2,\dots,\bar{t}\,\}$. Based on the development of \citet[Section 4.3]{YANG2016}, we obtain a customer arrivals model using a Poisson process with \textit{time-invariant} event rate $\lambda \in (0,1)$ for all $t\in T$ from a Poisson process with homogeneous time steps, but \textit{time-varying} event rate. 

Customers can choose from a number of (typically 1-hour wide) delivery time windows, which we call slots $s \in S$, where $S:=\{1,2,\dots,\bar{s}\}$. Let $s=0$ correspond to a customer not choosing any slot. Each delivery slot $s$ is assigned a delivery charge $d_{s} \in \left[\underline{d},\bar{d}\,\right] \cup \infty $, for some minimum allowable charge $\underline{d} \in \mathbb{R}$ (which is typically, though not necessarily, non-negative) and some maximum allowable charge $\bar{d} \geq \underline{d}$. The role of $d_{s}=\infty$ is a convention to indicate that slot $s$ is not offered. This is explained in more detail when introducing the customer choice model below. We define the delivery charge vector \mbox{$d:=[d_1,d_2,\dots,d_{\bar{s}}]^{\intercal}$}. Let the set of admissible delivery charge vectors be \mbox{$D:=\set{d}{d_{s} \in \left[\underline{d},\bar{d}\,\right] \text{ for all } s \in S}$}. 

For each delivery slot \mbox{$s \in S$}, we denote the number of placed orders by $x_s\in \mathbb{Z}$. We also define $x :=[x_1,x_2,\dots,x_{\bar{s}}]^{\intercal}\in \mathbb{Z}^{\bar{s}}$ as well as $X:=\set{x}{0\leq x_{s} \leq \bar{x}_{s} \text{ for all } s\in S}$, where $\bar{x}_{s}$ is a scalar indicating the maximum number of deliveries that can be fulfilled in slot $s$. In general, we do not require the maximum number of deliveries to be the same for all slots, e.g.\ since this will depend on traffic patterns in the delivery area. Examples of computing this quantity can be found in \citet[Section 4]{YANG2017}. Let us define $\bar{x}:=[\bar{x}_1,\bar{x}_2,\dots,\bar{x}_{\bar{s}}]^{\intercal}$ as well as the set of feasible slots $F(x)=\set{s\in S}{x+1_s \in X}$. Let $r \in \Re$ denote the expected net revenue of an order, i.e.\ expected revenue minus costs prior to delivery. This is assumed to be invariant across all orders. We define 
\begin{equation}\label{eq:cost}
C(x):=
\begin{dcases}
C_{\mathbb{\Re_+}}(x), & \text{if }x \in X \\
\infty & \text{otherwise,}
\end{dcases}
\end{equation}
where $C_{\Re_+}: X \to \Re_+$ is a given function. The function $C$ approximates the delivery cost to fulfil the set of orders $x$. The delivery cost cannot be computed exactly, as it is the solution to a vehicle routing problem with time windows, which is intractable for industry-sized applications \citep{TOTH2014}.

Let the probability that a customer chooses delivery slot $s$ if offered prices $d$ be $\Pi_{s}(d)$, such that $d \mapsto \Pi_{s}(d)\in[0,1)$ for all $s \in S$. Note that $\sum_{s \in S}\Pi_{s}(d) = 1-\Pi_{0}(d)$, where $\Pi_{0}>0$ denotes the probability of a customer leaving the online ordering platform without choosing any delivery slot. A typical choice for $\Pi_{s}$ is the multinomial logit model that was also used in \citet{YANG2017}:
\begin{equation}\label{eq:mnl}
\Pi_{s}(d):=\frac{\exp(\beta_c+\beta_s+\beta_d d_{s})}{\sum_{k\in S}\exp(\beta_c+\beta_k+\beta_dd_{k})+1},
\end{equation}
where $\beta_c \in \Re$ denotes a constant offset, $\beta_s \in \Re$ represents a measure of the popularity for all delivery slots and $\beta_d < 0$ is a parameter for the price sensitivity. Note that the no-purchase utility is normalised to zero, i.e.\ for the no-purchase ``slot'' $s=0$, we have a no-delivery ``charge'' $d_0=0$, such that \mbox{$\beta_c+\beta_0+\beta_d d_0 = \beta_c+\beta_0 = 0$} and hence, the $1$ in the denominator of \eqref{eq:mnl} arises from $\exp(\beta_c+\beta_0)=1$.

Note that our results on the fixed point computation do not depend on the particular form of the customer choice model. We require only that it is a probability distribution and that, for all $s \in S$ and as $d_{s} \to \infty$, we have $\Pi_{s}(d)$ tending to zero with a faster than linear rate of convergence; otherwise, the expected profit-to-go as defined below will be unbounded.

For convenience, let the probability that a customer arrives and chooses slot $s$ given prices $d$ be denoted by $p_{s}(d) :=\lambda\Pi_{s}(d)$. We define $p(d):=[p_1(d),p_2(d),\dots,p_{\bar{s}}(d)]^{\intercal}$ and $P:=\set{p(d)}{d\in D}$. Finally, it is to be understood that all sums over $s$ are always computed over the entire set $S$.
\subsection{Dynamic Programming Formulation}
We can express the problem described above as a DP. The expected profit-to-go closely resembles the DP formulation in \citet{YANG2017} and we define it as
\ifTwoColumn
\begin{equation}
\begin{split}\label{eq:rmA}
V_^A{t-1}(x)&:=\underset{d}{\max}\left\{\sum_{a,s}p_{a,s}(d)\left[r+d_{a,s}+V_^A{t}(x+1_{a,s}) \right.\right.\\
&\phantom{:=\;}\left.\left.-V_^A{t}(x)\right]\vphantom{\sum_{1}}+V_^A{t}(x) \right\}\text{ for all } x \in X, t\in T,\\
&\phantom{:=\;}\text{ where } V_{\bar{t}+1}^A(x) = -C(x)\text{ for all } x \in X,
\end{split}
\end{equation}
\else
\begin{equation}\label{eq:rm}
\begin{split}
V_t(x):=&\underset{d\in D}{\max}\left\{\sum_{s}p_{s}(d)\left[r+d_{s}+V_{t+1}(x+1_{s})-V_{t+1}(x)\right]+V_{t+1}(x)\right\},\\ &\text{for all } (x,t) \in X\times T,\text{ where } V_{\bar{t}+1}(x) = -C(x)\quad \forall x \in X,
\end{split}
\end{equation}
\fi
i.e.\ $C(\cdot)$ denotes the terminal condition. The difference $V_t(x)-V_t(x+1_{s})$ represents the value foregone by accepting an additional (discrete spatial) order, which in economic terms is the opportunity cost of an order. Note that -- similar to \citet{YANG2017} -- we ignore any vehicle load capacity constraints in the problem, as they are much less restricting than the time constraints on the delivery slots. Therefore, including the vehicle load capacity constraints would only increase computational costs, but would not substantially improve the decision policy. For convenience in the sequel, we define the DP operator $\mathcal{T}$ to express \eqref{eq:rm} in a more compact form as
\begin{equation}\label{eq:abstract_operator}
V_{t-1}:=\mathcal{T} V_{t}, \text{ for all } t\in T.
\end{equation}
\section{Infinite and Finite Time Horizon Results}\label{sec:res}

\subsection{Infinite Time Horizon Result}\label{subsec:inf_res}
We first consider the infinite horizon case, i.e.\ where $\bar{t} = \underset{\tau \to \infty}{\lim} \tau$.	In this scenario, we can find a fixed point of the DP described by \eqref{eq:abstract_operator} based on the following assumption.
\begin{assumption}\label{ass:profit_pos}
	The marginal cost of an additional, feasible order is always smaller than the maximum marginal profit, i.e.\ $C(x+1_{s})-C(x)\leq \bar{d}+r$, for all \mbox{$(x,s)\in X\times F(x)$}.
\end{assumption}
Assumption \ref{ass:profit_pos} is not restrictive, since it offers the means to ensure that every additional, feasible order can generate profit. Otherwise, the delivery slot prices, which maximise \eqref{eq:rm}, would always be $d_{s}=\infty$ for all $s\in S$, resulting in not offering any slots. Based on the aforementioned definitions and Assumption \ref{ass:profit_pos}, we formulate our first result, proof of which is deferred to Section \ref{sec:inf_proof}.
\begin{theorem}\label{th:fixedpoint}
	Under Assumption \ref{ass:profit_pos}, the unique fixed point of \eqref{eq:abstract_operator} is given by
	\begin{equation}
	\label{eq:dp_fixedpoint}
	V^{*}(x):= (\bar{d}+r)\mathbf{1}^\intercal (\bar{x}-x)-C(\bar{x}), \text{ for all } x\in X.
	\end{equation}	
\end{theorem}
There is a natural interpretation of this perhaps surprisingly compact result: The fixed point of the DP is a hyperplane in $x$, where each element of the gradient of $V^*$ is equal to $-(\bar{d}+r)$, or equivalently, the opportunity cost of an order is $V_t(x)-V_t(x+1_{s})=\bar{d}+r$ for all $(x,s) \in X \times F(x)$. Therefore, the only optimal selection of delivery slot prices is to choose $\bar{d}$ for all delivery slots $s \in F(x)$. For any other choice the opportunity costs would be larger than the revenue generated by any order. This result makes intuitive sense as in the limit as $t \to -\infty$, there will always arrive enough customers who will be willing to pay $\bar{d}$ for a delivery. Therefore, in the infinite time horizon case, it is best to always charge the maximum admissible delivery charge.
\subsection{Finite Time Horizon Result}\label{subsec:fin_res}
For finite $\bar{t}$, we establish a geometric property of the value function $V_t$, $t\in T$, related to concavity of a continuous function. As the domain of $V_t$ is discrete, it is not possible to establish this property from convexity theory. We provide some definitions before stating our main result. 

Let the opportunity cost of an order in slot $s$ at time $t$ be denoted by $\gamma_{s,t}(x):=V_t(x)-V_t(x+1_{s}) \geq 0$ for all $(x,s,t)\in X\times S\times T$. Let $\gamma_t(x):=[\gamma_{1,t}(x),\gamma_{2,t}(x),\dots,\gamma_{\bar{s},t}(x)]$. We define the set of stochastic vectors in $X$ as 
	\begin{equation}
	\mathcal{V}_X:=\set{v\in \Re_{+}^N}{\sum_{i= 1}^{\dim(X)}v_i=1},
	\end{equation}
where $v_i$ denotes the $i$-th component of $v$. Let $x\in X$ and let $Q\subseteq \mathbb{Z}^{\bar{s}}$ be a finite set. Then $Q$ is defined to be an \emph{enclosing set} of $x$ if $x \in$ conv$(Q)$. We define $\mathcal{Q}(x)$ as the set of all sets $Q$ enclosing $x$. The following two definitions are frequently used in discrete convex analysis:
\begin{definition}[cf.\ {\citet[(2.1)]{MUROTA2001}}]\label{de:conc_clo}
Let $a \in \Re^N$ and $b \in \Re$. Then the concave closure $\tilde{f}: \Re^N \to \Re \cup \{-\infty\}$ of a function $f:\mathbb{Z}^N \to \Re \cup \{-\infty\}$ is defined as
\begin{equation}\label{eq:concclo}
\tilde{f}(x):=\underset{a,b}{\inf}\set{a^\intercal x + b}{a^\intercal y + b \geq f(y) \hspace{4mm} \forall y \in \mathbb{Z}^N}.
\end{equation}
\end{definition}
\begin{definition}[{cf.\ \citet[Lemma 2.3]{MUROTA2001} and \cite[Proposition 2.31]{ROCKAFELLAR98}}]\label{de:conc_ext}
A function $f: \mathbb{Z}^N \to \Re \cup \{-\infty\}$ is concave-extensible if and only if any of the following equivalent conditions hold:
\begin{itemize}
\item[(a)] The evaluations of $f$ coincide with the evaluations of its concave closure $\tilde{f}$, i.e.\ $f(x)=\tilde{f}(x)$ for all $x\in \mathbb{Z}^N$.
\item[(b)] For all $x\in X$ and for all $Q\in \mathcal{Q}(x)$, the evaluation of $f$ at $x$ does not lie below any possible linear interpolation of $f$ on the points $q\in Q$, i.e.\ for all $x\in X$, for all $Q\in \mathcal{Q}(x)$ and for all $\mu \in \mathcal{V}_X$, such that $ x = \sum_{q\in Q}\mu_q q$, it holds that
\begin{equation} \label{eq:conc_ext}
f(x) \geq \sum_{q\in Q} \mu_q f(q).
\end{equation}
\end{itemize}
\end{definition}
Based on these definitions, we impose the following assumptions on our finite time horizon result.
\begin{assumption}\label{as:concavecost}
	We assume that the opportunity cost at the terminal condition $\gamma_{s,\bar
{t}+1}$ of all orders is increasing in $x$ for all unit hypercubes in $X$, i.e.\ $\gamma_{s,\bar{t}+1}(x) < \gamma_{s,\bar{t}+1}(x+1_{s'})$ for all $(x,s,s')\in X\times F(x)\times F(x)$, such that $s\neq s'$.
\end{assumption}
\begin{assumption}\label{ass:c}
	The function $-C$ is concave-extensible.
\end{assumption}
Assumption \ref{as:concavecost} is satisfied if $V_{\bar{t}+1}$ is strictly submodular. This is since for all strictly submodular functions $f: \mathbb{Z}^n \to \Re$ we have
\begin{equation}
f(\max(y,z))+f(\min(y,z)) < f(y)+f(z)
\end{equation}
for all $y$ and $z \in$ dom$(f)$, where the maximum and minimum are taken componentwise (e.g.\ see \citet[Definition 3.2]{BERTSIMAS2005}). This means that $f$ has increasing opportunity costs, since, for all $(x,s,s')\in X\times S\times S$, such that $s\neq s'$, we can set $f = V_{\bar{t}+1}$, $y=x+1_s$, $z=x+1_{s'}$, which yields the desired inequality:
\begin{equation}
\begin{split}
V_{\bar{t}+1}(x+1_s+1_{s'})+V_{\bar{t}+1}(x) &< V_{\bar{t}+1}(x+1_s) + V_{\bar{t}+1}(x+1_{s'}) \\
\iff V_{\bar{t}+1}(x)-V_{\bar{t}+1}(x+1_s) &< V_{\bar{t}+1}(x+1_{s'}) -V_{\bar{t}+1}(x+1_s+1_{s'}) \\
\iff \gamma_{s,\bar{t}+1}(x) &< \gamma_{s,\bar{t}+1}(x+1_{s'}).
\end{split}
\end{equation}
Since $V_{\bar{t}+1}$ needs to be strictly submodular, this requires that $C$ is strictly supermodular as $V_{\bar{t}+1}(x)=-C(x)$ for all $x\in X$. This is not the case for all $C$ used in the literature. For example, \citet{YANG2017} use an affine cost function. However, our results are also relevant for situations with affine cost functions, since -- as we show numerically in Section \ref{sec:ex} -- the value function can reach a state where Assumption \ref{as:concavecost} is satisfied in a small number of iterations of the Bellman operator. Assumption \ref{ass:c} is weak as it is satisfied by any convex cost function, which also includes the aforementioned affine cost functions. We can now state our second main result.
\begin{theorem}\label{th:v_conc}	
	Under Assumptions \ref{ass:profit_pos}, \ref{as:concavecost} and \ref{ass:c}, there exists a sufficiently small $\lambda > 0$, such that $V_t$ is finite-valued, concave-extensible in $x$ for all $t\in T$.
\end{theorem}
In the following section, we prove our main two results. Furthermore, we quantify a range of values for $\lambda$ such that Theorem \ref{th:v_conc} always holds. This condition is reported in \ref{ap:uncon}.
\section{Proofs of Main Results}\label{sec:proofs}
\subsection{Proof of Infinite Time Horizon Theorem}\label{sec:inf_proof}
To prove Theorem \ref{th:fixedpoint}, we first note that the DP in \eqref{eq:abstract_operator} can be reformulated as a so-called stochastic shortest path problem \citep[see][Section 4.1.1]{LEBEDEVETAL2019A}. The Bellman operator mapping of this class of problems is known to be contractive (see \citet[Chapters 1 and 3]{BERTSEKAS2012} and \citet[Lemma 5]{LEBEDEVETAL2019A}). Therefore, the DP in \eqref{eq:abstract_operator} admits a unique fixed point. We start with the necessary and sufficient condition for $\mathcal{T}$ to have a fixed point $V^{*}$, which is $V^{*}=\mathcal{T}V^{*}$. Setting $V_t(x)=V_{t+1}(x)=V^*(x)$ in \eqref{eq:rm} yields
\ifTwoColumn
\begin{equation}\label{eq:dp_fixedpoint_cond}
\begin{split}
0&=\underset{d}{\max}\left\{\sum_{s}p_{s}(d)[r+d_{s}\right.\\
&\phantom{=\;}\left.+V^*(x+1_{s})-V^{*}(x)]\vphantom{\sum_a}\right\}.
\end{split}
\end{equation}
\else
\begin{equation}\label{eq:dp_fixedpoint_cond}
\underset{d\in D}{\max}\left\{\sum_{s}p_{s}(d)[r+d_{s}+V^{*}(x+1_{s})-V^{*}(x)]\right\}= 0.
\end{equation}
\fi
Substituting the candidate $V^*$ from \eqref{eq:dp_fixedpoint} into \eqref{eq:dp_fixedpoint_cond} results then in
\begin{equation}
\label{dp:fixedpoint_finalcond} \underset{d\in D}{\max} \left\{\sum_{s}p_{s}(d)[d_{s} -\bar{d}\,]\right\} = 0.
\end{equation}
The values of all $p_{s}(d)$ are non-negative for all \mbox{$s\in S$} and for all $d \in D$. The value of $[d_{s}-\bar{d}\,]$ is non-positive and $0$ only if $d_{s}=\bar{d}$ for all $s \in S$. It follows that the maximum non-negatively weighted sum of the $[d_{s}-\bar{d}]$ terms is $0$, so \eqref{dp:fixedpoint_finalcond} holds. Finally, notice that $V_t(\bar{x)})=C(\bar{x})$ for all $t\in T$. Since the candidate fixed points satisfies \mbox{$V^{*}(\bar{x})=V_{\bar{t}+1}(\bar{x})=-C(\bar{x})$}, $V^{*}$ is a fixed point of $\mathcal{T}$ for all $x \in X$. \qed
\subsection{Proof of Finite Time Horizon Theorem}\label{sec:fin_proof}
In this section, we prove Theorem \ref{th:v_conc}. We start by reformulating \eqref{eq:rm} as a maximisation over $p \in P$ instead of $d \in D$. As shown by \citet{DONGETAL2009}, this is possible, since, for all $s\in S$, the following unique mapping between $p$ and $d$ exists:
\begin{equation}
\frac{p_{s}}{p_{0}}=\exp(\beta_c+\beta_s+\beta_d d_{s}),
\end{equation}
where we recall that $p_0=\lambda \Pi_0>0$. We solve this equation with respect to $d_{s}$ to obtain
\begin{equation}\label{eq:d_ito_p}
d_{s}=\beta_d^{-1}\left[\ln\left(\frac{p_{s}}{p_{0}}\right)-\beta_c-\beta_s\right].
\end{equation}
We will prove the theorem by induction. To this end, we fix an arbitrary $t \in T$, assume for an induction hypothesis that $V_t$ is concave-extensible in $x$ and now show that $V_{t-1}=\mathcal{T}V_t$ is also concave-extensible. Note that the base case in our induction proof is captured by Assumption \ref{ass:c}. By substituting \eqref{eq:d_ito_p} into \eqref{eq:rm} we obtain
\begin{equation}\label{eq:fandg}
\begin{split}
&\begin{split}
\mathcal{T}V_t(x)&= \underset{p\in P}{\max}\sum_{{s}}p_{s} \left\{r+\beta_d^{-1}\left[\ln\left(\frac{p_{s}}{p_{0}}\right)-\beta_c-\beta_s\right]\right. \\
&\phantom{=\;}+V_t(x+1_{{s}})-V_t(x)\vphantom{\sum_a}\left.\vphantom{\frac{a}{a}}\right\}+V_t(x)
\end{split}\\
&\phantom{\mathcal{T}V_t(x)}=\underset{p\in P}{\max} \{f(p)+g_t(x,p)\},
\end{split}
\end{equation}
where we have defined
\begin{equation}\label{eq:f_def}
\begin{split}
f(p):=& \sum_{{s}}p_{s}\left\{ r+\beta_d^{-1}\left[\ln\left(\frac{p_{s}}{p_{0}}\right)-\beta_c-\beta_{s}\right]\right\},\\
g_t(x,p):=& \sum_{s}p_{s}\left\{V_t(x+1_{{s}})-V_t(x)\right\}+V_t(x)
\end{split}
\end{equation}
for all $(x,p) \in X\times P$. This allows us to formulate the following result, whose parts we prove in Appendices A and B, respectively. Recall that $p_{s} =\lambda\Pi_{s}$ for all $s \in S$.
\begin{lemma}\label{le:conc}
For all $t\in T$, the functions $f$ and $g_t$ have the following properties:
\begin{enumerate}[(i)]
\item The function $f$ is concave in $p$.
\item Under Assumption \ref{as:concavecost} and if $V_t$ is concave-extensible in $x$, there exists a sufficiently small $\lambda > 0$ such that the function $g_t$ is concave-extensible in $(x,p)$.
\end{enumerate}
\end{lemma}
The proof of Lemma \ref{le:conc}(i) is given in \ref{ap:concproof1}. In Lemma\ref{le:conc}(ii), we assume that $V_t$ is concave-extensible in $x$ as this is embedded within our induction proof for Theorem $\ref{th:v_conc}$, where this corresponds to our induction hypothesis. The proof of Lemma \ref{le:conc}(ii) depends on the following self-contained result. Let us consider a relaxation of the constraint on the optimisation variable $d$ in \eqref{eq:rm} and optimise over $\Re^{\bar{s}}$ instead of $D$. We refer to this problem as the \textit{unconstrained} DP as opposed to the original, \textit{constrained} DP.
\begin{proposition}\label{pr:oppcost} The constrained and unconstrained version of the DP share the following property:
\begin{enumerate}[(i)]
\item	Consider the \textit{unconstrained} DP. Under Assumption \ref{as:concavecost}, the opportunity cost $\gamma_{s,t}$ of all orders is increasing for all unit hypercubes in $X$, i.e.\
	\begin{equation}\label{eq:gammainc}
	\gamma_{s,t}(x+1_{s'}) > \gamma_{s,t}(x)
	\end{equation}
	for all $(x,s,s',t)\in X\times F(x)\times F(x) \times T$, such that $s\neq s'$.
\item Property \eqref{eq:gammainc} also holds for the \textit{constrained} DP. 
\end{enumerate}
\end{proposition}
We prove the two parts of Proposition \ref{pr:oppcost} in \ref{ap:uncon} and \ref{ap:con}, respectively. These results are then used in \ref{ap:concproof} to prove Lemma \ref{le:conc}(ii). An interesting implication of Proposition \ref{pr:oppcost} is that increasing opportunity costs in $x$ imply that the optimal pricing strategy will be increasing in $x$. This is due to the unique optimal prices being increasing in opportunity costs \citep{DONGETAL2009}.

By Lemma \ref{le:conc}, there exists a sufficiently small $\lambda$ such that $g_t$ has a continuous extension $\tilde{g}_t$, which is jointly concave in $(x,p)$. By inspection, $f$ is only a function of the continuously-valued variable $p$. Therefore, $f(p)+\tilde{g}_t(x,p)$ is also jointly concave in $(x,p)$. We define $U(x):=\underset{p\in P}{\max} \, \{f(p)+\tilde{g}_t(x,p)\}$. By \citet[Proposition 2.22]{ROCKAFELLAR98} or \citet[Section 3.2.5]{BOYD2004}), partial maximisation with respect to some variables of a continuous multivariate function that is jointly concave in all its variables, preserves concavity in the resulting function. Therefore $U$ is a concave function of $x$.

It remains to show that $U(x)=\mathcal{T}V_t(x)$ for all gridpoints $x\in X$. Repeating the same calculation, now with the discrete $f(p)+g_t(x,p)$ in place of $f(p)+\tilde{g}_t(x,p)$, i.e.\ \mbox{$\mathcal{T}V_t(x)=\underset{p\in P}{\max} \, \{f(p)+g_t(x,p)\}$}, note that $\tilde{f}(p)+\tilde{g}_t(x,p)=f(x)+g_t(x,p)$ for all $x\in X$ by Definition \ref{de:conc_ext}(a). Therefore, $\mathcal{T}V_t(x)=U(x)$ for all $x\in X$. This shows that $U$ is a continuous extension of $\mathcal{T}V_t$, which is concave in $x$. Hence, $\mathcal{T}V_t$ is concave-extensible in $x$. This concludes our induction argument and shows that the value function $V_t$ is concave-extensible in $x$ for all $t\in T$. \qed
\section{Illustrative Example}\label{sec:ex}
We illustrate our findings using a simple numerical example of a 2-slot problem. The parameters are listed in Table \ref{tab:parameters} below.
\begin{table}[H]
	\caption{The parameters of the numerical example.}
	\centering
	\begin{tabular}{r|l}\label{tab:parameters}
		$\lambda$ & $0.5$ \\
		$\bar{t}$ & $200$ \\
		$\left[\underline{d},\bar{d}\,\right]$ & $[0,2]$ \\
		$r$ & 1 \\
		$S$ & $\{1,2\}$ \\
		$\bar{x}$ & $[4,4]^{\intercal}$ \\
		$[\beta_c,\beta_d,\beta_{1},\beta_{2}]$ & $[1,-1,1,-1]$ \\
		$C_{\Re_+}(x)$ & $2+x_1+2x_2$	
	\end{tabular}
\end{table}
Notice that $C$ violates Assumption \ref{as:concavecost}. However, as our numerical example shows, after a few iterations of the Bellman operator, the opportunity costs become strictly increasing by inspection, so that our results still hold. The parameters yield the terminal condition 
\begin{equation}
V_{\bar{t}+1}(x):=-C(x)=-2-x_1-2x_2.
\end{equation} 
From Theorem \ref{th:fixedpoint}, we obtain the fixed point
\begin{equation}
\begin{split}
V^*(x)&:=(\bar{d}+r)\mathbf{1}^\intercal (\bar{x}-x)-C(\bar{x})\\
&\phantom{:}=10-3(x_1+x_2)
\end{split}
\end{equation} 
for all $x\in X$. We define -- as a measure of discrete concavity -- the quantity
\begin{equation}
\epsilon_t:=\underset{x,Q\in \mathcal{Q}(x)}{\min}\left\{V_t(x)-\sum_{q\in Q} \mu_q V_t(q)\right\}
\end{equation}
for all $t \in T$, such that $\mu \in \mathcal{V}$ and $\sum_{q\in Q} \mu_q V_t(q)=V_t(x)$ for all $t\in T$. In comparison with \eqref{eq:conc_ext}, notice that $\epsilon_t \geq 0$ implies that $V_t$ is concave-extensible. We compute $\epsilon_t$ for all $t\in T$ by enumeration of all possible enclosing sets and plot the result in Fig.\ \ref{fig:numex}(a), from which we see that $\epsilon_t\geq 0$ for all $t\in T$.
\begin{figure}[H]
	\centering
	\begin{subfigure}[t]{0.48\textwidth}
		\centering
		\includegraphics[width=\columnwidth,trim=0cm -1.5cm 0cm 1cm,clip]{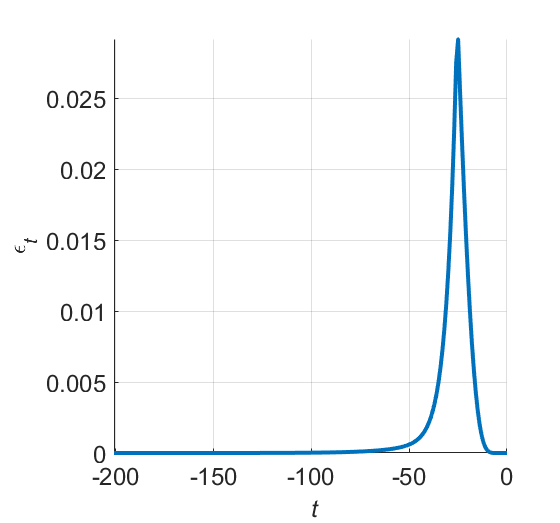}
		\caption{The concave extensibility measure $\epsilon(t)$ is non-negative for all time steps $t$ in the booking horizon.}
		\label{fig:gap}
	\end{subfigure}
	\hfill
	\begin{subfigure}[t]{0.45\textwidth}
	\centering
	\includegraphics[width=\columnwidth,trim=5cm 7.7cm 4.5cm 7.9cm,clip]{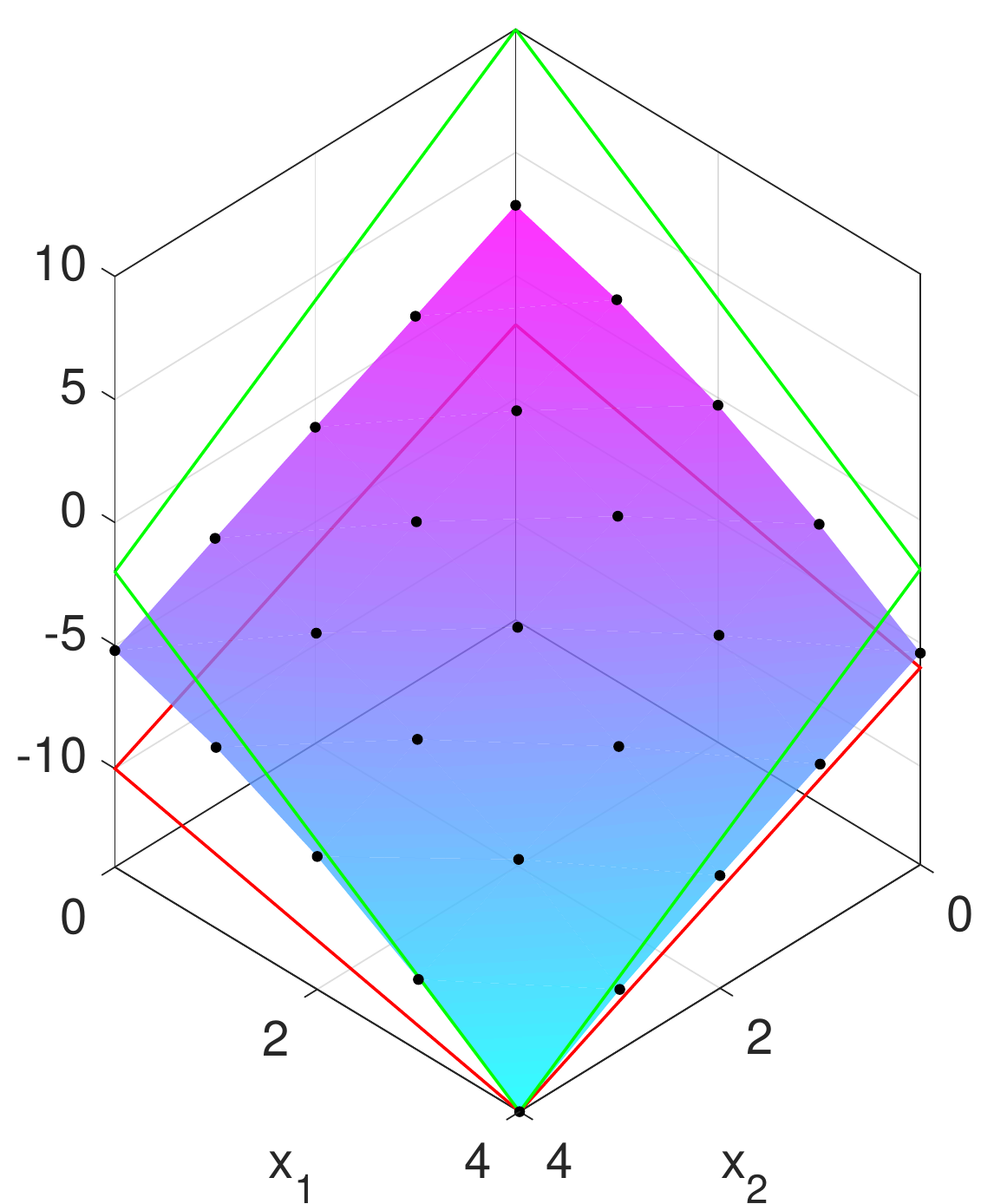}
	\caption{The value function at $t=\bar{t}+1$ (red), at the fixed point $t=-\infty$ (green) and at $t=\bar{t}-10$ (blue/violet).}
	\label{fig:value_func}
	\end{subfigure}
	\caption{Illustrative example of a 2-slot problem.}
	\label{fig:numex}
\end{figure}
Finally, we plot $V_t$ at the terminal condition $t=\bar{t}+1$, at the fixed point $t=-\infty$ and at $t=\bar{t}-10$ in Fig.\ \ref{fig:numex}(b). When it comes to approximating $V_t$, we can use the observation that $V_t$ always lies between the terminal condition and the fixed point to limit the range of basis function parameters, such that the approximated version of $V_t$ also falls between these lower and upper bounds. Also notice that the value function at $t=\bar{t}-10$ is concave-extensible and has increasing opportunity costs.
\section{Conclusions and Future Work}\label{sec:conc}
We have studied the mathematical properties of the value function of a dynamic program modelling the revenue management problem in attended home delivery exactly. We have shown that the recursive dynamic programming mapping has a unique, finite-valued fixed point and concavity-preserving properties. Hence, we have derived our main result stating that -- under certain assumptions -- for all time steps in the dynamic program, the value function admits a continuous extension, which is a finite-valued, concave function of its state variables. We have illustrated our findings on a simple numerical example. 

Recent approaches have estimated $V_t$ as an affine function of $x$ for each $t \in T$ \citep{YANG2017}. Based on our result, we believe that closer approximations can be found by pursuing different approximation strategies. One such strategy would be to adapt approximate DP algorithms from stochastic dual dynamic programming, also known as SDDP. The idea is to use a cutting plane algorithm to successively form tighter upper bounds to the value function described as the point-wise minimum of affine functions \citep{PEREIRA1991,SHAPIRO2011}.

A second possible direction of future research involves investigating the use of parametric models comprising concave basis functions. This idea can be exploited directly by using the given DP formulation -- as suggested in \citet[Section 8.2]{POWELL2007} -- or by reformulating the problem as a linear program -- as shown by \citet{DEFARIAS2003}. Note that \textit{a priori} knowledge of concave extensibility of $V_t$ for all $t\in T$ creates some intuitive regularity. Therefore, it can be expected to get good approximations of $V_t$ from a relatively small sample size even with simple models.

Another possible direction would be to adapt techniques that fit convex functions (or equivalently concave functions for our purposes) to multidimensional data. For example, \citet{KIM2004} and \citet{MAGNANI2009} show how data can be fitted by a function defined as the maximum of a finite number of affine functions. More sophisticated examples of convex (concave) function fitting techniques include adaptive partitioning \citep{HANNAH2013} and Bayesian non-parametric regression \citep{HANNAH2011}.

\section*{Acknowledgements}
This work was supported by SIA Food Union Management. The authors are grateful for this financial support.
\appendix

\section{Proof of Lemma \ref{le:conc}(i)}\label{ap:concproof1}
It is shown in \citet{DONGETAL2009} that a structurally similar function to $f$ is concave in its variables. We adopt a similar approach -- computing the Hessian and showing that it is negative definite -- to verify that $f$ is jointly concave in all components of the vector $p$. We first compute the first-order partial derivatives of $f$:
\begin{equation}
\begin{split}
\frac{\partial f}{\partial p_{i}}&=[r+\beta_d^{-1}(\ln(p_{i}/p_{0})-\beta_c-\beta_i)]+\beta_d^{-1}, \text{ for all } i\in S,\\
\frac{\partial f}{\partial p_{0}}&=\sum_{s} -p_{s}\beta_d^{-1}p_{0}^{-1}.
\end{split}
\end{equation}
The second-order partial derivatives are:
\begin{equation}
\begin{split}
\frac{\partial^2 f}{\partial p_{i}^2}&=\beta_d^{-1}p_{i}^{-1}, \text{ for all } i \in S,\\
\frac{\partial^2 f}{\partial p_{0}^2}&=\sum_{s}p_{s}\beta_d^{-1}p_{0}^{-2},\\
\frac{\partial^2 f}{\partial p_{i} \partial p_{0}}&=-\beta_d^{-1}p_{0}^{-1}, \text{ for all } i\neq 0,\\
\frac{\partial^2 f}{\partial p_{i} \partial p_{j}} &= 0, \text{ for all } (i,j)\in S\times S, \text{ such that } i\neq j.
\end{split}
\end{equation}
The resulting Hessian $H$ of $f$ with its second partial derivatives with respect to $\{p_{i}\}$ for all $i \in S\cup \{0\}$ is:
\begin{equation}
\begin{split}
H&:=\beta_d^{-1}\begin{bmatrix}
p_{1}^{-1} & \dots & 0 & -p_{0}^{-1} \\
\vdots & \ddots &\vdots & \vdots \\
0 & \dots & p_{\bar{s}}^{-1} & -p_{0}^{-1} \\
-p_{0}^{-1} & \dots & -p_{0}^{-1} & p_{0}^{-2}\sum_{s}p_{s}
\end{bmatrix} \\
&\hphantom{:}=:\begin{bmatrix}
A & B \\
B^\intercal & C
\end{bmatrix},
\end{split}
\end{equation}
where we have defined block sub-matrices $A,B,B^\intercal$ and $C$ of appropriate dimension, such that $C$ is a scalar corresponding to the last entry of $H$. Note that $A$ is negative definite, because $p_{s}\geq0$ for all $s\in S \cup \{0\}$ and $\beta_d^{-1}<0$. We compute the Schur complement of $A$ in $H$:
\begin{equation}
C-B^{\intercal}A^{-1}B=\beta_d^{-1}\left(\sum_{s}p_{s} p_{0}^{-2}-p_{0}^{-2}\sum_{s}p_{s}\right)=0.
\end{equation}
As a result of $\beta_d <0$, $A$ is negative definite and as $H/A$ is non-positive, $H$ is negative semi-definite (see e.g.\ \citet[Appendix A.5.5]{BOYD2004}). This implies that $f$ is concave in $p$.\qed

\section{Proof of Lemma \ref{le:conc}(ii)}\label{ap:concproof2}
The proof of Lemma \ref{le:conc}(ii) requires several intermediate results which we present in the following sections before returning to the proof of Lemma \ref{le:conc}(ii).
\subsection{Auxiliary function definitions and properties}
In this section, we define some auxiliary functions and establish some of their properties that are needed in the subsequent sections. We define $W: \Re_+ \mapsto \Re_+$ as the inverse function of $f: \Re_+ \mapsto \Re_+$, such that $f(x):=x\exp(x)$, i.e.\ implicitly defined through the relationship
\begin{equation}
x = W(x)\mathrm{e}^{W(x)}.
\end{equation}
Note that $W$, as defined above, is the principal branch of the so-called Lambert $W$ function, which is uniquely defined over the non-negative real numbers. To simplify notation in the following proof, we define two more functions: We define $\psi_s: \Re_+ \mapsto \Re_{++}$ as 
\begin{equation}\label{eq:psi}
\psi_s(z):= \exp(\beta_c+\beta_s+\beta_d(z-r)-1)
\end{equation}
for all $s\in S$. We define the function $\phi: \Re_{+}^{\bar{s}} \mapsto \Re_{++}$ as
\begin{equation}\label{eq:phi}
\phi(z):=-\lambda \beta_d^{-1}W\left(\sum_s\psi_s(z_s)\right),
\end{equation}
where $z_s$ indicates the $s$-th component of $z$. We can now establish some properties of $\phi$ that are instrumental for the subsequent proof of Lemma \ref{le:conc}(ii).

\begin{lemma}\label{le:phi}
The function $\phi$ has the following properties:
\begin{enumerate}[(i)]
\item It is decreasing over its domain.
\item It satisfies the inequality:
\begin{align}\label{eq:min_procedure}
&\phi(\gamma_t(x+1_{s'})) -\phi(\max\{\gamma_t(x+1_{s}),\gamma_t(x+1_{s'})\}) + \phi(\gamma_t(x+1_{s}))\nonumber \\
\geq& \phi(\min\{\gamma_t(x+1_{s}),\gamma_t(x+1_{s'})\})
\end{align}
for all $(x,s,s')\in X\times S\times S$, such that $s\neq s'$.
\end{enumerate}
\end{lemma}

\begin{proof}
\begin{enumerate}[(i)]
\item It is useful to state the first derivative of $W$, which is
\begin{equation}\label{eq:W1stderi}
\frac{\mathrm{d}W}{\mathrm{d}y}(y)=\frac{W(y)}{y(1+W(y))}.
\end{equation}
It suffices to show that the first partial derivative of $W$ with respect to the components $z_i$ for all 
$i \in S$ is negative. To this end, fix any $i \in S$. Setting $y=\sum_s \psi_s(z_s)$, gives
\begin{equation}\label{eq:1stpartderi}
\begin{split}
\frac{\partial \phi}{\partial z_i}(z) =& -\lambda \beta_d^{-1} \frac{\mathrm{d} W}{\mathrm{d}y}(y)\frac{\partial y}{\partial z_i}(z) \\
=&-\lambda\beta_d^{-1}\frac{W(\sum_s \psi_s(z_s))}{\left[\sum_s \psi_s(z_s)\right](1+W\left(\sum_s \psi_s(z_s)\right)}\\
&\times \beta_d \exp(\beta_c+\beta_i+\beta_d (z_i-r)-1)\\
=&-\lambda \frac{W(\sum_s \psi_s(z_s))}{1+W\sum_s \psi_s(z_s)}\frac{\psi_i(z_i)}{\sum_s \psi_s(z_s)},
\end{split}
\end{equation}
where the last equality follows from the definition of $\psi_i$ in \eqref{eq:psi}. The customer arrival rate $\lambda \in (0,1)$. The first fraction in \eqref{eq:1stpartderi} lies in $(0,1)$ as $W(y)\geq 0$ for all $y\in $dom$(W)$, while the second one lies in $(0,1]$ as $\psi_s(y) > 0$ for all $y \in$ dom$(\psi_s)$ for all $s\in S$ and $\bar{s} \geq 1$. Therefore, $\partial \phi / \partial z_i (z) \in (-1,0)$ for all $i\in S$ and hence, it is negative.

\item Fix any $i \in S$. Let $\alpha_{i,t}:=\max\{\gamma_{i,t}(x+1_{s}),\gamma_{i,t}(x+1_{s'})\}$ and $\beta_{i,t}:=\min\{\gamma_{i,t}(x+1_{s}),\gamma_{i,t}(x+1_{s'})\}$ and notice that $\alpha_{i,t} \geq \beta_{i,t}$. Let us distinguish two cases.

\emph{Case I:} Suppose that $\gamma_{i,t}(x+1_{s})\geq \gamma_{i,t}(x+1_{s'})$. For all $j\in S, j \neq i$, define $\epsilon_{j,t}^\alpha:=\gamma_{j,t}(x+1_{s})$, $\epsilon_{j,t}^\beta:=\gamma_{j,t}(x+1_{s'})$ and $\epsilon_{j,t}^{\alpha\beta}:=\gamma_{j,t}(x+1_{s}+1_{s'}).$ Under these assignments, the left-hand side of \eqref{eq:min_procedure} can be equivalently written as
\begin{align}\label{eq:phi_epsilon}
&\phi(\gamma_t(x+1_{s'})) -\phi(\max\{\gamma_t(x+1_{s}),\gamma_t(x+1_{s'})\}) + \phi(\gamma_t(x+1_{s}))\nonumber\\
=&\phi(\epsilon_{1,t}^\beta,\dots,\beta_{i,t},\dots,\epsilon_{\bar{s},t}^\beta) -\phi(\epsilon_{1,t}^{\alpha\beta},\dots,\alpha_{i,t},\dots,\epsilon_{\bar{s},t}^{\alpha\beta}) \nonumber\\
&+ \phi(\epsilon_{1,t}^{\alpha},\dots,\alpha_{i,t},\dots,\epsilon_{\bar{s},t}^{\alpha}).
\end{align}
Define the scalar function $f_\theta: A \mapsto \Re$, where $A$ contains all real numbers no smaller than $\beta_{i,t}$ and $\theta:=\{\{\epsilon_{j,t}^\alpha,\epsilon_{j,t}^\beta, \epsilon_{j,t}^{\alpha\beta}\}_{j\neq i},\beta_{i,t}\}$, such that
\begin{equation}
\begin{split}
f_\theta(\alpha_{i,t})=\;& \phi(\epsilon_{1,t}^\beta,\dots,\beta_{i,t},\dots,\epsilon_{\bar{s},t}^\beta)\\
&- \phi(\epsilon_{1,t}^{\alpha\beta},\dots,\alpha_{i,t},\dots,\epsilon_{\bar{s},t}^{\alpha\beta})\\
&+ \phi(\epsilon_{1,t}^{\alpha},\dots,\alpha_{i,t},\dots,\epsilon_{\bar{s},t}^{\alpha}).
\end{split}
\end{equation}
Consider the derivative of $f_\theta$, given by
\begin{equation}
\frac{\mathrm{d}f_\theta}{\mathrm{d}\alpha_{i,t}} (\alpha_{i,t}) = -\frac{\partial\phi}{\partial \alpha_{i,t}}(z^{\alpha\beta})+\frac{\partial\phi}{\partial \alpha_{i,t}}(z^{\alpha}),
\end{equation}
where we have defined $z^{\alpha\beta} := [\epsilon_{1,t}^{\alpha\beta},\dots,\alpha_{i,t},\dots,\epsilon_{\bar{s},t}^{\alpha\beta}]$ as well as $z^{\alpha} := [\epsilon_{1,t}^{\alpha},\dots,\alpha_{i,t},\dots,\epsilon_{\bar{s},t}^{\alpha}]$. We compute the derivative of $\phi$ from the first equation in \eqref{eq:1stpartderi} to arrive at
\begin{equation}
\frac{\mathrm{d}f_\theta}{\mathrm{d}\alpha_{i,t}} (\alpha_{i,t}) = \lambda \beta_d^{-1} \frac{\mathrm{d} W}{\mathrm{d}y}(y^{\alpha\beta})\frac{\partial y^{\alpha\beta}}{\partial \alpha_{i,t}}(z^{\alpha\beta}) - \lambda \beta_d^{-1} \frac{\mathrm{d} W}{\mathrm{d}y}(y^{\alpha})\frac{\partial y^{\alpha}}{\partial \alpha_{i,t}}(z^{\alpha}),
\end{equation}
where $y^{\alpha\beta}:=\sum_s \psi_s(z_s^{\alpha\beta})$ and  $y^{\alpha}:=\sum_s \psi_s(z_s^{\alpha})$. Substituting for the derivatives of $y^{\alpha\beta}$ and $y^{\alpha}$ and simplifying yields
\begin{equation}
\frac{\mathrm{d}f_\theta}{\mathrm{d}\alpha_{i,t}} (\alpha_{i,t}) =\lambda\psi_i(\alpha_{i,t})\left[\frac{\mathrm{d}W}{\mathrm{d}z}(z^{\alpha\beta})-\frac{\mathrm{d}W}{\mathrm{d}z}(z^{\alpha})\right] \geq 0.
\end{equation}
The inequality follows from noting that $z^\alpha \geq z^{\alpha\beta}$ and that $W$ has a negative second derivative over its domain. We conclude that $f_{\theta}$ is non-decreasing in $\alpha_{i,t}$, which means that $f_{\theta}$ is non-increasing by decreasing $\alpha_{i,t}$ to its minimum value $\alpha_{i,t} = \beta_{i,t}$. Repeating this minimisation for all $i\in S$, we obtain the following bound:
\begin{align}
& \phi(\gamma_t(x+1_{s'})) -\phi(\max\{\gamma_t(x+1_{s}),\gamma_t(x+1_{s'})\}) + \phi(\gamma_t(x+1_{s})) \nonumber\\
\geq\;& \phi(\min\{\gamma_t(x+1_{s}),\gamma_t(x+1_{s'})\}) - \phi(\min\{\gamma_t(x+1_{s}),\gamma_t(x+1_{s'})\})\nonumber\\
&+ \phi(\min\{\gamma_t(x+1_{s}),\gamma_t(x+1_{s'})\})\nonumber\\
=\;& \phi(\min\{\gamma_t(x+1_{s}),\gamma_t(x+1_{s'})\}),
\end{align}
as required. 

\emph{Case II:} Suppose that the roles of $s$ and $s'$ are now reversed, i.e.\ $\gamma_{i,t}(x+1_{s'})\geq \gamma_{i,t}(x+1_{s})$. Via symmetric arguments, we reach the same conclusion as in Case I.

As both cases reach the same conclusion and collectively exhaust all possibilities, this concludes our proof and shows that \eqref{eq:min_procedure} holds.
\end{enumerate}
\end{proof}
\subsection{Proof of Proposition \ref{pr:oppcost}(i)}\label{ap:uncon}
We start by reformulating the DP in \eqref{eq:rm} in terms of $\phi$. Fix any $t \in T\setminus\{1\}\cup\{\bar{t}+1\}$. The unique optimisers of the unconstrained optimisation problem at time $t-1$, denoted by $d_{s}^*(x)$ for all $(x,s)\in X\times S$, is given by \citet{YANG2017} based on the development of \citet{DONGETAL2009} as
\begin{equation}\label{eq:d_star}
d_{s}^*(x) = \gamma_{s,t}(x) - r - \beta_d^{-1}h_t(x)
\end{equation}
for all $s \in S$, where $h_t(x)$ is the unique solution of
\begin{equation}\label{eq:h}
(h_t(x)-1)\exp\left(h_t(x)\right) = \sum_s \exp\left(\beta_c + \beta_s + \beta_d(\gamma_{s,t}(x) - r)\right).
\end{equation}
We rewrite \eqref{eq:h} equivalently as
\begin{align}
\iff (h_t(x)-1)\exp\left(h_t(x)-1\right) =& \sum_s \exp\left(\beta_c + \beta_s + \beta_d(\gamma_{s,t}(x) - r)-1\right)\nonumber\\
\iff (h_t(x)-1)\exp\left(h_t(x)-1\right) =& \sum_s \psi_s(\gamma_{s,t}(x)),
\end{align}
where we have used $\psi_s$ from \eqref{eq:psi}. By the definition of $W$, we obtain
\begin{equation}
h_t(x) = 1 + W\left(\sum_s \psi_s(\gamma_{s,t}(x))\right).\label{eq:h_ito_W}
\end{equation}
Now, we can substitute \eqref{eq:h_ito_W} into \eqref{eq:d_star}:
\begin{equation}\label{eq:d_star_ito_W}
d_{s}^*(x) = \gamma_{s,t}(x) - r - \beta_d^{-1}\left[1 + W\left(\sum_s \psi_s(\gamma_{s,t}(x))\right)\right].
\end{equation}
Finally we can substitute \eqref{eq:d_star_ito_W} into the unconstrained version of \eqref{eq:rm} to obtain
\begin{equation}\label{eq:dpitoW}
\begin{split}
\mathcal{T}V_{t}(x)=& \sum_s p_{s}(d^*(x))\left\{r+\gamma_{s,t}(x) - r - \beta_d^{-1}\left[1 + W\left(\sum_{s'} \psi_s(\gamma_{s',t}(x))\right)\right]\right.\\
&\left.\vphantom{\sum_a}-\gamma_{s,t}(x)\right\}+V_t(x)\\
=& \sum_s p_{s}(d^*(x))\left\{- \beta_d^{-1}\left[1 + W\left(\sum_{s'} \psi_s(\gamma_{s',t}(x))\right)\right]\right\}+V_t(x),
\end{split}
\end{equation}
where we have defined $d^*(x):=[d_1^*(x),d_2^*(x),\dots,d_{\bar{s}}^*(x)]^{\intercal}$. We now substitute the customer choice model $p$ evaluated at the optimiser $d^*(x)$  into \eqref{eq:dpitoW}:
\begin{equation}\label{eq:unc_stage_opt_2}
\begin{split}
\mathcal{T}V_{t}(x)=\;& \sum_s \frac{\lambda\exp(\beta_c+\beta_s+\beta_d d_{s}^*(x) )}{\sum_{s''}\exp(\beta_c+\beta_{s''}+\beta_d d_{s''}^*(x))+1}\\
& \times \left\{- \beta_d^{-1}\left[1 + W\left(\sum_{s'} \psi_s(\gamma_{s',t}(x))\right)\right]\right\}+V_t(x).
\end{split}
\end{equation}
Note that -- using the definitions of $d_{s}^*$ and $h_t$ -- the following relationship holds:
\begin{equation}\label{eq:util_ito_h}
\begin{split}
&\sum_s \exp(\beta_c+\beta_s+\beta_d d_{s}^*(x))\\
=&\sum_s \exp\left(\beta_c+\beta_s+\beta_d \left(\gamma_{s,t}(x)-r-\beta_d^{-1}h_t(x)\right)\right)\\
=&\sum_s \exp\left(\beta_c+\beta_s+\beta_d \left(\gamma_{s,t}(x)-r\right)\right)\exp(-h_t(x))\\
=&\left(h_t(x)-1\right)\exp(h_t(x))\exp(-h_t(x))\\
=& h_t(x)-1,
\end{split}
\end{equation}
where the third equality follows from \eqref{eq:h}. Substituting \eqref{eq:util_ito_h} into \eqref{eq:unc_stage_opt_2}, we obtain
\begin{equation}
\mathcal{T}V_{t}(x)=\lambda \frac{h_t(x)-1}{h_t(x)}\left\{- \beta_d^{-1}\left[1 + W\left(\sum_s \psi_s(\gamma_{s,t}(x))\right)\right]\right\}+V_t(x).
\end{equation}
Substituting for $h_t(x) $ using \eqref{eq:h_ito_W} yields the desired expression:
\begin{align}\label{eq:dp_phi}
\mathcal{T}V_{t}(x)=\;&\frac{\lambda W\left(\sum_s \psi_s(\gamma_{s,t}(x))\right)}{1+W\left(\sum_s \psi_s(\gamma_{s,t}(x))\right)}\left\{- \beta_d^{-1}\left[1 + W\left(\sum_s \psi_s(\gamma_{s,t}(x))\right)\right]\right\}\nonumber\\
&+V_t(x)\nonumber\\
=\;& - \lambda \beta_d^{-1} W\left(\sum_s \psi_s(\gamma_{s,t}(x))\right)+V_t(x)\nonumber\\
=\;&\phi(\gamma_t(x))+V_t(x),
\end{align}
where the last inequality follows from the definition of $\phi$ in \eqref{eq:phi}. By Assumption \ref{as:concavecost}, the terminal condition of the DP satisfies the property of increasing opportunity costs.

We can now prove Proposition \ref{pr:oppcost}(i) by showing that a monotonic mapping exists between 
$\gamma_{s,t-1}(x+1_{s'})-\gamma_{s,t-1}(x)$ and $\gamma_{s,t}(x+1_{s'})-\gamma_{s,t}(x)$ for all $(x,s,s',t)\in X\times F(x) \times F(x) \times (T\cup\{\bar{t}+1\}\setminus\{1\})$.

To this end, fix any $(x,s,t) \in X\times F(x) \times (T\cup\{\bar{t}+1\}\setminus\{1\})$. By using the definition of opportunity costs and \eqref{eq:dp_phi} we can write the opportunity cost in state $x$ with respect to an arbitrary slot $s\in S$ at time $t-1$ as
\begin{equation}\label{eq:opp_cost_t-1}
\gamma_{s,t-1}(x)=\gamma_{s,t}(x)+\phi(\gamma_t(x))-\phi(\gamma_t(x+1_{s})).
\end{equation}
Fix any $s'\in F(x)$, such that $s'\neq s$. To prove the theorem, we require
\begin{equation}\label{eq:opp_cost_diff}
\gamma_{s,t-1}(x+1_{s'}) - \gamma_{s,t-1}(x) > 0
\end{equation}
for all $t\in T\cup\{\bar{t}+1\}\setminus\{1\}$. Substitute \eqref{eq:opp_cost_t-1} into the left-hand side of \eqref{eq:opp_cost_diff} to obtain
\begin{equation}
\begin{split}
&\gamma_{s,t-1}(x+1_{s'})-\gamma_{s,t-1}(x) \\
=\; & \gamma_{s,t}(x+1_{s'})-\gamma_{s,t}(x)+\phi(\gamma_t(x+1_{s'}))\\ 
&-\phi(\gamma_t(x+1_{s}+1_{s'}))- \phi(\gamma_t(x))+ \phi(\gamma_t(x+1_{s})).
\end{split}
\end{equation}
First note that $\gamma_{t}(x+1_{s}+1_{s'}) \geq \gamma_t(x+1_{s})$ and $\gamma_{t}(x+1_{s}+1_{s'}) \geq \gamma_t(x+1_{s'})$. As $\phi(\gamma_{t}(x+1_{s}+1_{s'}))$ is decreasing in its argument by Lemma \ref{le:phi}(i) and since it is subtracted on the right-hand side of the above equation, we can create the following lower bound:
\begin{equation}\label{eq:target_ineq}
\begin{split}
& \gamma_{s,t-1}(x+1_{s'})-\gamma_{s,t-1}(x) \\
\geq\; & \gamma_{s,t}(x+1_{s'})-\gamma_{s,t}(x)+\phi(\gamma_t(x+1_{s'}))\\ & -\phi(\max\{\gamma_t(x+1_{s}),\gamma_t(x+1_{s'})\})
- \phi(\gamma_t(x))+ \phi(\gamma_t(x+1_{s})).
\end{split}
\end{equation}
Using Lemma \ref{le:phi}(ii), we bound \eqref{eq:target_ineq} from below by
\begin{align}\label{eq:4terms}
\gamma_{s,t-1}(x+1_{s'})-\gamma_{s,t-1}(x) \geq\; & \phi(\min\{\gamma_t(x+1_{s}),\gamma_t(x+1_{s'})\})-\phi(\gamma_t(x)) \nonumber\\
&+\gamma_{s,t}(x+1_{s'})-\gamma_{s,t}(x),
\end{align}
where the minimum is taken element-wise. Since $\phi$ is both positive by definition, we construct a lower bound on \eqref{eq:4terms} by dropping the $\phi(\min\{\gamma_t(x+1_{s}),\gamma_t(x+1_{s'})\})$ term on the right-hand side. Furthermore, since $\phi$ is decreasing in its argument by Lemma \ref{le:phi}(i), we create another lower bound on \eqref{eq:4terms} by setting $\phi(\gamma_t(x))=\mathbf{0}$. Hence, we obtain
\begin{equation}
\begin{split}
&\gamma_{s,t-1}(x+1_{s'})-\gamma_{s,t-1}(x) \\
>\;& -\phi(\mathbf{0}) +\gamma_{s,t}(x+1_{s'})-\gamma_{s,t}(x)\\
=\;& \lambda\beta_d^{-1} W\left(\sum_s\psi_s(0)\right) +\gamma_{s,t}(x+1_{s'})-\gamma_{s,t}(x).
\end{split}
\end{equation}
Since only $\beta_d^{-1}$ is negative, $\lambda\beta_d^{-1} W\left(\sum_s\psi_s(0)\right)<0$, independent of the choice of $(s,s',t,x)$. Therefore, the above inequality describes a monotonic mapping from $\gamma_{s,t}(x+1_s')-\gamma_{s,t}(x)$ to $\gamma_{s,t-1}(x+1_s')-\gamma_{s,t-1}(x)$ for all $(s,s',t-1,x)\in F(x)\times F(x) \times T \times X$, such that $s\neq s'$. The bound on $\gamma_{s,t}(x+1_s')-\gamma_{s,t}(x)$ will therefore decrease as $t$ decreases. Hence, $\gamma_{s,t}(x+1_s')-\gamma_{s,t}(x)$ will be minimal at $t=1$. Using the monotonicity of this mapping, we can find a $\lambda$ for which $\gamma_{s,1}(x+1_s')-\gamma_{s,1}(x) > 0 $ by repetitively applying the above equation starting from the terminal condition at $t=\bar{t}+1$ to obtain
\begin{equation}
\gamma_{s,1}(x+1_{s'})-\gamma_{s,1}(x) \geq \bar{t}\lambda\beta_d^{-1} W\left(\sum_s\psi_s(0)\right) +\gamma_{s,\bar{t}+1}(x+1_{s'})-\gamma_{s,\bar{t}+1}(x),
\end{equation}
where the right-hand side is positive if
\begin{equation}\label{eq:lambda*}
\begin{split}
0&<\bar{t}\lambda\beta_d^{-1} W\left(\sum_s\psi_s(0)\right) +\gamma_{s,\bar{t}+1}(x+1_{s'})-\gamma_{s,\bar{t}+1}(x)\\
\iff \lambda &< -\beta_d \frac{\gamma_{s,\bar{t}+1}(x+1_{s'})-\gamma_{s,\bar{t}+1}(x)}{\bar{t}W\left(\sum_s\psi_s(0)\right)}
\end{split}
\end{equation}
for all $(x,s,s')\in X\times F(x)\times F(x) $, such that $s\neq s'$. The right-hand side of the above expression is strictly positive, because opportunity costs are increasing at the terminal condition by Assumption \ref{as:concavecost} and therefore, the numerator of the fraction is positive, $W$ only takes positive values, $\bar{t}>0$ and $\beta_d<0$. As both $X$ and $F(x)$ are finite sets, a small enough $\lambda>0$ can be found that satisfies all inequalities described by \eqref{eq:lambda*}. Therefore, a $\lambda > 0$ exists, such that $\gamma_{s,t}(x+1_{s'}) > \gamma_{s,t}(x)$ for all $(x,s,s',t)\in X\times F(x)\times F(x) \times T$, such that $s\neq s'$.
\qed

\subsection{Proof of Proposition \ref{pr:oppcost}(ii)}\label{ap:con}
Fix any $(x,t) \in X \times T$. Define $u:=\gamma_{t}(x)$ and consider the function $w: \Re_+^{\bar{s}} \mapsto \Re^{\bar{s}}$ defined by
\begin{equation}
w(u):=u-r-\beta_d^{-1}\left[1+W\left(\sum_s \psi_s(u)\right)\right].
\end{equation}
In comparison with \eqref{eq:d_star_ito_W}, notice that $w(u)$ is mathematically identical to $d^*(x)=[d_1^*(x),d_2^*(x),\dots,d_{\bar{s}}^*(x)]$. Furthermore, $W$ and $\psi_s$ for all $s\in S$ are invertible functions. In particular, their inverses in the domain of interest are
\begin{equation}
\begin{split}
x =& W(x)\exp\left(W(x)\right) \text{ and}\\
u =& \left(\ln\left(\psi_s(u)\right)-\beta_c-\beta_s+1\right)\beta_d^{-1}+r
\end{split}
\end{equation}
for all $s \in S$. Therefore, $w$ is a composition of invertible functions and hence, the mapping between $w$ and $u$ is bijective. Therefore, the mapping between $d^*(x)$ and $\gamma_{t}(x)$, equivalent to the mapping between $w$ and $u$, is bijective. 

If we constrain $d^*(x)$ to $D \subset \Re^{\bar{s}}$, we can conclude that there still exists a bijective mapping between $D$ and (an unknown) $\Gamma \subset \Re_+^{\bar{s}}$ corresponding to the range of values that $\gamma_t(x)$ can take in this constrained scenario. Conversely, this means that no matter which $d^*(x) \in D$ maximises the \textit{constrained} stage optimisation problem, there exists $\hat{\gamma}_t(x)\in \Gamma$, which is linked to the same $d^*(x)$ in the \textit{unconstrained} problem. In other words, there exists $\hat{\gamma}_t(x)\in \Gamma$, which produces the same $V_{t-1}(x)$ in the \textit{unconstrained} problem as $\gamma_t(x)$ does in the \textit{constrained} problem. Due to this equivalence, the following statement is a necessary and sufficient condition for Proposition \ref{pr:oppcost}(i) to hold: Evaluating the \textit{unconstrained} problem at $\hat{\gamma}_t(x)$, i.e.\ $V_{t-1}(x) = \phi(\hat{\gamma}_t(x))$ for all $x\in X$, there exists a sufficiently small $\lambda >0$ that yields a value function at time $t-1$, whose opportunity cost is increasing in $x$. From Proposition \ref{pr:oppcost}(i), we know that this statement holds true for all opportunity costs $\hat{\gamma}_{t}(x)\in \Re_+^{\bar{s}}$ under Assumption \ref{as:concavecost}. Since $\hat{\gamma}_t(x)\in \Gamma \subset \Re_+^{\bar{s}}$, there exists a sufficiently small $\lambda>0$ such that the opportunity cost will also be increasing in the \textit{constrained} DP.\qed

\subsection{Completing the proof of Lemma \ref{le:conc}(ii)}\label{ap:concproof}
From Definition \ref{de:conc_ext}(b), $g_t$ is concave-extensible in $(x,p)$ if
\begin{equation}\label{eq:gtarget}
g_t(x,p) \geq \sum_{q \in Q}\mu_q g_t\left(q^{(x)},q^{(p)}\right)
\end{equation}
for all $(x,p)\in X\times P$ and for all enclosing sets $Q$, such that $[x,p]^{\intercal}=\sum_{q \in Q}\mu_q \left[q^{(x)},q^{(p)}\right]^\intercal$. We show that this inequality holds by starting from the right-hand side and substituting for $g_t(q^{(x)},q^{(p)})$:
\begin{equation}\label{eq:g_conc_ineq}
\begin{split}
\sum_{q \in Q}\mu_q g_t(q^{(x)},q^{(p)})=&\sum_{q \in Q}\mu_q \left[ \sum_{s}q_{s}^{(p)}\left\{V_t(q^{(x)}+1_{s})-V_t(q^{(x)})\right\}+V_t(q^{(x)})\right]\\
=&\sum_{q \in Q}\mu_q \left[ \sum_{s}q_{s}^{(p)}V_t(q^{(x)}+1_{s})+\big(1-\sum_{s}q_{s}^{(p)}\big)V_t(q^{(x)})\right],
\end{split}
\end{equation}
where we note that each of the summed terms in square brackets is a convex combination on the set $A(q^{(x)}):=\{q^{(x)} \cup \{q^{(x)}+1_s\}_{s\in F(q^{(x)})}\}$, where the set of indices $s$ is $F(q^{(x)})$ instead of $S$, since $q_s^{(p)}=0$ for all $s\notin F(q^{(x)})$, i.e.\ assigning zero transition probability to infeasible slots. To show that \eqref{eq:gtarget} holds, we now derive the supporting result that there exists a small enough $\lambda>0$, such that the concave closure $\tilde{V}_t(y)$, evaluated at any point $y \in X$, is a hyperplane on the set $A(y)$, which will make it possible to simplify \eqref{eq:g_conc_ineq} further. 

Fix any $y\in X$ and consider the unit hypercube in the positive direction of $y$, which we define as $B(y):=\set{z}{y\leq z \leq y+\sum_{s\in F(y)}1_s}$. We will show that only points in $A(y)$ form the concave closure around $y$ by demonstrating that for all $y'\in B(y)\setminus A(y)$, the line segment between $(y,V_t(y))$ and $(y',V_t(y'))$ lies below a second line segment between two other points $(z,V_t(z))$ and $(z',V_t(z'))$ for some $(z,z')\in B(y)\times B(y)$, i.e.\ we will show that there exists some $(z,z')\in B(y)\times B(y)$, such that
\begin{equation}\label{eq:lineseg}
\alpha V_t(y)+(1-\alpha)V_t(y') < \beta V_t(z)+(1-\beta)V_t(z'),
\end{equation}
where $0\leq \alpha \leq 1 \text{ and }  0\leq \beta \leq 1$, such that $\alpha y + (1-\alpha) y' = \beta z + (1-\beta) z'$ for all $y'\in B(y)\setminus A(y)$. This means that the line segments cannot be part of the concave closure $\tilde{V}_t$. We show this result by induction on $|F(y)|$. Consider the base case when $|F(y)|=1$, then \eqref{eq:lineseg} is trivially satisfied as there is only a single element $s\in F(y)$, meaning that the set $B(y)\setminus A(y)=\emptyset$. Let $n:=|F(y)|$. Suppose by means of an induction hypothesis that \eqref{eq:lineseg} holds for all cardinalities of $F(y)$ up to and including $n-1$. Then the only line segment that we need to consider is the one connecting $y$ and $y'=y+\sum_{s\in F(y)}1_s$, because otherwise, we are in a lower-dimensional case, for which \eqref{eq:lineseg} holds by the induction hypothesis. For this choice of $(y,y')$, we can find a quadruple $(z,z',\alpha,\beta)$ that satisfies \eqref{eq:lineseg} by repetitively invoking Proposition \ref{pr:oppcost}(ii) as follows: There exists a sufficiently small $\lambda>0$, such that 
\begin{equation}
\begin{split}
&\gamma_{1}(y)\\
&<\gamma_{1}(y+1_{2})\\
&<\gamma_{1}(y+1_{2}+1_{3})\\
&\vdots\\
&<\gamma_{1}\big(y+\sum_{s\in \sigma}1_s\big),
\end{split}
\end{equation}
where we have defined $\sigma:= F(y)\setminus\{1\}$, which -- strictly speaking -- depends on $y$ and $s=1$, but we neglect this to ease notation. We can expand the first and last line of the above chained inequality by using the definition of opportunity costs:
\begin{equation}
\begin{split}
\gamma_{1}(y)&< \gamma_{1}\big(y+\sum_{s\in \sigma}1_s\big)\\
\iff V_t(y)-V_t(y+1_{1})&< V_t\big(y+\sum_{s\in \sigma}1_s\big)-V_t\big(y+\sum_{s\in F(y)}1_s\big)\\
\iff V_t(y)-V_t(y+1_{1})&< V_t\big(y+\sum_{s\in \sigma}1_s\big)-V_t(y'),
\end{split}
\end{equation}
where we have used $y'=y+\sum_{s\in F(y)}1_s$ as previously. Rearranging and multiplying both sides by $1/2$ yields
\begin{equation}\label{eq:lineineq}
\iff \frac{1}{2}V_t(y)+\frac{1}{2}V_t(y')< \frac{1}{2}V_t\big(y+\sum_{s\in \sigma}1_s\big)+\frac{1}{2}V_t(y+1_{1}).
\end{equation}
Comparing with \eqref{eq:lineseg}, we see that $\alpha=\beta=1/2$, $z=y+\sum_{s\in \sigma}1_s$ and $z'=y+1_{1}$. To illustrate this for the case when $n=3$, we plot the points $y,y',y+\sum_{s\in \sigma}1_s$ and $y+1_1$ as well as the line segments between the former and latter pair of points in Fig.\ \ref{fig:linesinbox} below.

\begin{figure}[H]
	\centering
	\includegraphics[width=0.7\columnwidth,trim=2cm 2cm 2cm 2cm,clip]{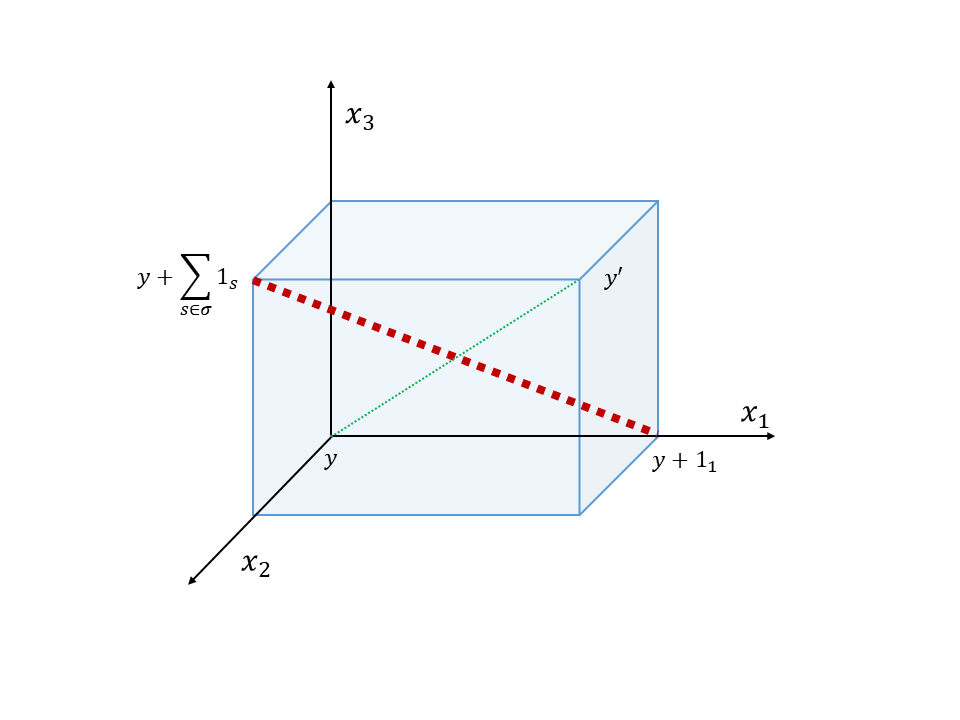}
	\caption{Thin, green, dotted line segment between $y$ and $y'$ and bold, red, dotted line segment between $y+\sum_{s\in \sigma}1_s$ and $y+1_1$. By \eqref{eq:lineineq} we have that at the intersection of the two lines, the interpolation of $V_t$ between the two extreme points of the bold, red line lies above the interpolation of $V_t$ between the two extreme points of the thin, green line. The light blue, solid line box spans all the points that are contained in $B(y)$.}
	\label{fig:linesinbox}
\end{figure}

Since we have found a quadruple $(z,z',\alpha,\beta)$ that satisfies \eqref{eq:lineseg}, we conclude that no $y'\in B(y)\setminus A(y)$ can be part of the concave closure $\tilde{V}_t$ for all $y\in B(y)$ and for all $n$. Therefore, only points in $A(y)$ can be part of the concave closure $\tilde{V}_t$ for all $y\in B(y)$ and, in fact, they all are, since the $|F(y)|+1$ points in $A(y)$ uniquely define the support vectors of a hyperplane through $\left\{(y,V_t(y))\right\}_{y \in A(y)}$, i.e.\ each $y\in A(y)$ giving rise to one linearly independent equality constraint, such that all $|F(y)|$ gradients and the offset of the hyperplane are uniquely defined. This holds true for unit hypercubes $B(y)$ for all $y \in X$ and hence, there exists a small enough $\lambda > 0$, such that $\tilde{V}_t$ is a hyperplane on the set $A(y)$ for all $y \in X$.

Returning to \eqref{eq:g_conc_ineq}, we notice that the expression in square brackets is a convex combination on the set $A(q^{(x)})$, which means that $\tilde{V}_t$ is a hyperplane on this set. Hence, we can rewrite \eqref{eq:g_conc_ineq} with equality as

\begin{align}
\sum_{q \in Q}\mu_q g_t(q^{(x)},q^{(p)})=& \sum_{q \in Q}\mu_q \tilde{V}_t\left(\sum_{s}q^{(p)}_{s}1_{s}+q^{(x)}\right)\nonumber\\
\leq&\tilde{V}_t\left(\sum_{q \in Q} \mu_q\left[ \sum_{s}q^{(p)}_{s}1_{s}+q^{(x)}\right]\right)\nonumber\\
=&\tilde{V}_t\left( \sum_{s} 1_{s} \sum_{q \in Q} \mu_q q^{(p)}_{s}+x\right)\nonumber\\
=&\tilde{V}_t\left( \sum_{s} 1_{s}p_{s}+x\right)\nonumber\\
=&\sum_{s}p_{s}\left\{V_t(x+1_{{s}})-V_t(x)\right\}+V_t(x)\nonumber\\
=& g_t(x,p),
\end{align}
where the second-last equality follows from the observation that the convex combination of $V_t$ is evaluated on a set $A(x)$, for which $\tilde{V}_t$ is again a hyperplane, and the inequality is obtained by noticing that $\tilde{V}_t$ is concave in $x$ and therefore, Jensen's inequality holds. From the above derivation, we conclude that $ g_t(x,p) \geq \sum_{q \in Q}\mu_q g_t(q^{(x)},q^{(p)})$, as required. Therefore, there exists a sufficiently small $\lambda>0$, for which $g_t$ is concave-extensible in $(x,p)$ if $V_t$ is concave-extensible in $x$.\qed
\section*{References}
\bibliography{RevConcaveValueFun04}
\end{document}